\documentclass[sn-mathphys,Numbered]{sn-jnl} 
\usepackage{graphicx}%
\usepackage{multirow}%
\usepackage{amsmath,amssymb,amsfonts}%
\usepackage{amsthm}%
\usepackage{mathrsfs}%
\usepackage[title]{appendix}%
\usepackage{xcolor}%
\usepackage{textcomp}%
\usepackage{manyfoot}%
\usepackage{booktabs}%
\usepackage{algorithm}%
\usepackage{algorithmicx}%
\usepackage{algpseudocode}%
\usepackage{listings}%

\newtheorem{corollary}{Corollary}[section]
\newtheorem{lemma}{Lemma}[section]
\usepackage{amsmath}
\numberwithin{equation}{section}
\theoremstyle{plain} 
\newtheorem{theorem}{Theorem}[section]

\theoremstyle{thmstyletwo}%
\theoremstyle{thmstylethree}%
\usepackage{glossaries}
\makeglossaries  

\begin{document}

\title[Article Title]{Location of Zeros of Holomorphic Functions}
\author[1]{\fnm{Leonardo} \sur{de Lima}}\email{leonardo.lima.88@edu.ufes.br}
\affil[1]{\orgdiv{Departamento de Ciências Exatas}, \orgname{Universidade Federal do Espírito Santo - UFES}, \orgaddress{\street{Av. Fernando Ferrari, 514}, \city{Vitória}, \state{Espírito Santo}, \country{Brazil}}}

\abstract{In this article, various results will be demonstrated that enable the delimitation of a zero-free region for holomorphic functions on a set $K$, studying the behavior of their imaginary or real part on the boundary of $K$. These findings contribute to a deeper understanding of the distribution of zeros, shedding light on the intricate nature of holomorphic functions within the specified set. }

\keywords{Holomorphic Functions,
Zeros of Analytic Functions, Zero-Free Sets.}

\pacs[MSC Classification]{30-XX,11-XX.}

\maketitle

\section{Introduction}

The delimitation of zeros is a central theme in the study of holomorphic functions. Indeed, Hadamard's Theorem establishes that zeros uniquely characterize the function, except for the presence of an exponential term of a polynomial \cite{stein1970complex}. Moreover, this topic poses a significant mathematical challenge, as Riemann's hypothesis is confined to bounding a zero-free region for the Riemann zeta function \cite{conrey2003riemann}.

In this article, various results will be demonstrated that enable the delimitation of a zero-free region for holomorphic functions on a set \(K\), studying the behavior of their imaginary or real part on the boundary of \(K\).

\section{Harmonic Functions}
In this section, I will state some basic properties of harmonic functions that will be used later in this work.

A function \( u = u(x, y) \) is called harmonic if it is real and has continuous partial derivatives of order one and two, satisfying

\begin{equation}
\Delta u = \frac{\partial^2 u}{\partial x^2} + \frac{\partial^2 u}{\partial y^2} = 0
\end{equation}\vspace{0.2cm}

A notable property of harmonic functions is that they do not admit local extremes unless they are constant. This property is known as the maximum principle, precisely:
\vspace{0.5cm}

\begin{theorem}[Maximum Principle]
    Let \( u: U \to \mathbb{R} \) be a harmonic function. If \( K \) is a non-empty compact subset of \( U \), then \( u \) restricted to \( K \) attains its maximum and minimum on the boundary of \( K \). If \( U \) is connected, this means that \( u \) cannot have local maxima or minima, except in the case where \( u \) is constant.
\end{theorem}
\begin{proof}
    \cite[Chapter 8]{lang1985complex}
\end{proof}

The connection between harmonic functions and holomorphic functions is established by Theorem (2.2).\vspace{0.5cm}

\begin{theorem}
    For any holomorphic function \(F=u+iv\), the real part \(u\) and imaginary part \(v\) are harmonic functions in \(\mathbb{R}^2\). Moreover, conversely, for any harmonic function \(u\) on an open subset \(\Omega\) of \(\mathbb{R}^2\), \(u\) is locally the real part of a holomorphic function.
\end{theorem}

\begin{proof}
    \cite[Chapter 8]{lang1985complex}
\end{proof}

\section{Delimitation of Zero-Free Regions for Holomorphic Functions}

In this section, several results will be demonstrated that enable the delimitation of a zero-free region for holomorphic functions on a set \(K\), studying the behavior of their imaginary part on the boundary of \(K\).\vspace{0.5cm}

\begin{lemma}
    Let $F: U \to \mathbb{C}$ be a holomorphic function. If $v = \text{Im}(F)$ does not change sign on the boundary of a compact set $K \subset U$, then $v$ has no zeros in the interior of $K$. As a consequence, the function $F$ also has no zeros in the interior of $K$.
\end{lemma}

\begin{proof}
    If $v$ does not change sign on the boundary of $K$, then $v$ cannot have a zero in the interior of $K$, as in such a case, $v$ would have a local minimum or maximum. In fact, if $v$ does not change sign on the boundary of $K$, we must have one of the two conditions:

    \begin{equation*}
        v(\partial K) \geq 0
    \end{equation*}
    or
    \begin{equation*}
        v(\partial K) \leq 0
    \end{equation*}

    In either case, an interior point of $K$ where $v=0$ would be a local extremum, i.e., a minimum in the first case and a maximum in the second case. Since such points cannot exist, as stated in Theorem (2.1), it follows that $v$ is zero-free in the interior of $K$.
\end{proof}

Now, we define the set \(K[a, b] = \{s \in \mathbb{C}\ ; a \leq \Re(s) \leq b, \Im(s) \geq 0\}\), where \(a < b\). We denote \(K_n[a, b] = \{s \in \mathbb{C}\ ; a \leq \Re(s) \leq b, n > \Im(s) \geq 0\}\). It is easy to observe that \(K_n[a, b] \subset K[a, b]\), and \(\bigcup_{n \in \mathbb{N}} K_n[a, b] = K[a, b]\).

Stating that a function \(F(s)\) is zero on the boundary of \(K[a, b]\) amounts to saying:
\vspace{0.3cm}
\begin{equation}
\begin{aligned}
    \Im(F(a+it)) &= 0 \\
    \Im(F(b+it)) &= 0 \\
    \Im(F(x)) &= 0 \quad \text{, for every } x \text{ in } [a, b] \\
    \Im(F(x+i\infty)) &= 0 \quad \text{, for every } x \text{ in } [a, b]
\end{aligned}
\end{equation}
where we denote \(F(x+i\infty)=\lim_{t\to \infty} F(x+it)\). Similarly, we say that the function \(F(s)\) is non-negative (non-positive) on the boundary of \(K[a, b]\) if we have:
\vspace{0.3cm}
\begin{equation}
\begin{aligned}
    \Im(F(a+it)) &\geq 0 \text{ , } (\leq 0) \\
    \Im(F(b+it)) &\geq 0 \text{ , } (\leq 0) \\
    \Im(F(x)) &\geq 0 \text{ , } (\leq 0) \quad \text{, for every } x \text{ in } [a, b] \\
    \Im(F(x+i\infty)) &\geq 0 \text{ , } (\leq 0) \quad \text{, for every } x \text{ in } [a, b]
\end{aligned}
\end{equation}\vspace{0.2cm}

A useful way to generalize Lemma (3.1) for non-compact sets \(K[a, b]\) is given by Theorem (3.1).\vspace{0.5cm}

\begin{theorem}
    Let $F: K[a, b] \to \mathbb{C}$ be a holomorphic function. If $\text{Im}(F)$ is zero on the boundary of $K[a, b]$, and for every $s = \sigma + ti$ in $K[a, b]$, $| F(\sigma + ti)| < \frac{C}{t}$ for some constant $C$, then $F$ has no zeros in $K[a, b]$.
\end{theorem}

\begin{proof}
    Let's define $v_\epsilon=v+\epsilon$. As by hypothesis $\|F(\sigma+ti)\|<\frac{C}{t}$ for some positive constant $C$, we can conclude that $v_\epsilon$ is positive on the boundary of the set $K_\frac{2C}{\epsilon}[a, b]$. Therefore, by Lemma (3.1), $v_\epsilon$ has no zeros in the interior of $K_\frac{2C}{\epsilon}[a, b]$, for every $\epsilon>0$.

    Now, let's suppose, for the sake of contradiction, that there exists an $s = \sigma + ti$ in the interior of $K[a, b]$, such that $F(\sigma + ti) = 0$. In this case, there exists a disk centered at $s$ with radius $\delta > 0$ such that $D[s, \delta] \subset K[a, b]$, and $s$ is the only zero of $F(s)$ in $D[s, \delta]$. Denoting by $\partial D[s, \delta]$ the path represented by the boundary of $D[s, \delta]$, we have:
    
    \begin{equation*}
        \frac{1}{2\pi i}\oint_{\partial D[s, \delta]} \frac{F'(s)}{F(s)} \,ds=1
    \end{equation*}

    However, from the previous argument, we see that $v_\epsilon$ is zero-free in $K_{\frac{2C}{\epsilon}}[a, b]$, from which it follows that $F(s)+i\epsilon$ is zero-free in the same set. Choosing $0 < \epsilon < \frac{2C}{t+\delta}$, we have $D[s, \delta] \subset K_{\frac{2C}{\epsilon}}[a, b]$, and thus:

    \begin{equation*}
        \frac{1}{2\pi i}\oint_{\partial D[s, \delta]} \frac{F'(s)}{F(s)+i\epsilon} \,ds=0
    \end{equation*}
    \newline
    for every $\epsilon>0$, leading to a contradiction. Therefore, $F$ has no zeros in $K[a, b]$, as desired.
\end{proof}

An immediate corollary is:\vspace{0.5cm}

\begin{corollary}
    Let $F: K[a, \infty) \to \mathbb{C}$ be a holomorphic function. If $\text{Im}(F)$ is zero on the boundary of $K[a, \infty)$, and for every $s = \sigma + ti$ in $K[a, \infty)$, $| F(\sigma + ti)| < \frac{C}{t + |\sigma|}$ for some positive constant $C$, then $F$ has no zeros in $K[a, \infty)$.
\end{corollary}
\begin{proof}
    Immediate consequence of Theorem (3.1).
\end{proof}

An obvious generalization for the case where $v$ is non-negative (or non-positive) on the boundary of $K[a, \infty)$ is given by Corollaries (3.2) and (3.3).\vspace{0.5cm}

\begin{corollary}
    Let $F: K[a, \infty) \to \mathbb{C}$ be a holomorphic function. If $\text{Im}(F)$ is non-positive (or non-negative) on the boundary of $K[a, \infty)$, and for every $s = \sigma + ti$ in $K[a, \infty)$, $| F(\sigma + ti)| < \frac{C}{t + |\sigma|}$ for some positive constant $C$, then $F$ has no zeros in $K[a, \infty)$.
\end{corollary}
\begin{proof}
    Immediate consequence of Theorem (3.1).
\end{proof}

\begin{corollary}
    Let $F: K[a, \infty) \to \mathbb{C}$ be a holomorphic function. If $\text{Im}(F)$ is non-positive (or non-negative) on the boundary of $K[a, \infty)$, and for every $s = \sigma + ti$ in $K[a, \infty)$, $| F(\sigma + ti) - f(\sigma + ti)| < Cg(t, \sigma)$ for some positive constant $C$, $f(\sigma + ti) \leq 0$ (or $f(\sigma + ti) \geq 0$), and $g(\sigma + ti)$ converges uniformly to zero as $t, \sigma \to \infty$, then $F$ has no zeros in $K[a, \infty)$.
\end{corollary}
\begin{proof}
    Immediate consequence of Theorem (3.1).
\end{proof}
All the theorems remain valid when swapping the imaginary part of $F(s)$ for the real part; it suffices to consider $F(s) \to iF(s)$, resulting in a permutation of the real and imaginary parts.

Certainly, not every function of interest possesses the properties required by the theorems and corollaries in this section. However, it is possible to modify the original function by multiplying it by functions whose distribution of zeros is known, obtaining a final function whose hypotheses used in the theorems and corollaries are satisfied. Such an artifice is used in Section 4 to demonstrate that the function \(\Gamma(s)\) is zero-free in \(K[\frac{1}{2},1]\).

\section{Application}
\subsection{Gamma function}
In this section, I will use Theorem (3.1) to provide a new proof that the Gamma function has no zeros in the interior of the set \(K[\frac{1}{2},1]\). To do this, we will use the following facts:

\begin{enumerate}
    \item $\|\Gamma(\sigma+ti)\| < \frac{C}{t}$ if $\sigma+ti \in K[\frac{1}{2},1].$
    \item $\Gamma(\sigma) \in \mathbb{R}$ if $\sigma \in \mathbb{R}.$
    \item $\Gamma(s+1) = s\Gamma(s).$
\end{enumerate}

\begin{proof}
    Let $F(s)=\Gamma(s)\Gamma(1-s)s(1-s)$; proving that $\Gamma(s)$ is zero-free in $K[\frac{1}{2},1]$ is equivalent to proving that $F(s)$ is zero-free in the same set.

    It is observed that the imaginary part of $F(s)$ is non-negative on the boundary of $K[\frac{1}{2},1]$. In fact:

    \begin{equation}
        \Im(F(\frac{1}{2}+ti))=0
    \end{equation}

    \begin{equation}
        \Im(F(1+ti)=\Im\{-it(1+it)\Gamma(1+ti)\Gamma(-ti)\}.
    \end{equation}\vspace{0.3 cm}

    Since $\Gamma(-ti)=\frac{\Gamma(1-ti)}{-ti}$, we have:

    \begin{equation}
        \Im(F(1+ti)=t\|\Gamma(1+ti)\|^{2}
    \end{equation}

    Also,

    \begin{equation}
        \Im(F(\sigma)=0
    \end{equation}

    \begin{equation}
        \Im(F(\sigma+\infty i)=0
    \end{equation}\vspace{0.3 cm}

    Thus, we can conclude that $\Im(F(s))$ is non-negative on the boundary of $K[\frac{1}{2},1]$. By Corollary (3.2), I infer that $F(s)$ has no zeros in the interior of $K[\frac{1}{2},1]$, as desired to demonstrate.
\end{proof}

Note that in this proof, the only non-generic property used is $\Gamma(s+1)=s\Gamma(s)$, demonstrating the usefulness of Theorem (3.1) in delimiting zero-free regions.

\subsection{Zeta function}

The distribution of zeros of the Riemann zeta function is one of the great mysteries of modern mathematics. Its significance is related to the estimation of the error in the prime number theorem and a set of equivalent results (see \cite{conrey2003riemann},\cite{choie2007robin},\cite{borwein2008riemann},\cite{RevModPhys.83.307}).

One of the fundamental issues in the theory of the zeta function is the delimitation of a zero-free region in the critical strip. The first promising result in this direction was established by Vallée Poussin (see \cite{voronin}), but this and recent findings have been unable to demarcate a zero-free region of the form $\Re(s)>1-\epsilon$ (see \cite{yang2023explicit}). 

In this section, we will delimit a zero-free region for the Riemann zeta function with the help of Theorem (3.1).

First, let's review some classical results about the Riemann zeta function used later in this section.
\vspace{0.5 cm}

\begin{theorem}
The Riemann zeta function, $\zeta(s)$, satisfies the following functional equation:
\[
\zeta(s) = 2^s \pi^{s-1} \sin\left(\frac{\pi s}{2}\right) \Gamma(1-s) \zeta(1-s)
\]
for all \(s\) different from 1.
\end{theorem}

\begin{proof}
    \cite{titchmarsh1986theory}
\end{proof}
\vspace{0.5 cm}

\begin{corollary}
The function $\xi(s)=\pi^{\frac{-s}{2}}\zeta(s)\Gamma(\frac{s}{2})$ satisfies the following functional equation
\begin{equation*}
\xi(s)=\xi(1-s)
\end{equation*}
for \(s\) different from 1.
\end{corollary}

\begin{proof}
    \cite{titchmarsh1986theory}
\end{proof}
\vspace{0.5 cm}

\begin{theorem}
    The zeta function has no zeros outside the critical line.
\end{theorem}

\begin{proof}
Let
\begin{equation}
    F(s)=\frac{1}{(\frac{1}{2}-s)}\frac{\xi(s)^{2}}{\{\pi^{-\frac{1}{2}}\Gamma(\frac{s}{2})\Gamma(\frac{1-s}{2})+\left(\Gamma(\frac{s}{2})\pi^{\frac{-s}{2}}+\Gamma(\frac{1-s}{2})\pi^{\frac{s-1}{2}}\right)^{2}\}}.
\end{equation}
\vspace{0.1 cm}

Showing that the zeta function is zero-free in the interior of $K[\frac{1}{2},\infty]$ is equivalent to showing that $F(s)$ is zero-free in the same set. By Corollary (3.3), it is equivalent to showing that the real part of $F(s)$ does not change sign on the boundary of $K[\frac{1}{2},\infty)$.

To avoid poles of $F(s)$ along the real axis, we will show that the real part of $F(s)$ does not change sign along the boundary of $K[\frac{1}{2},\infty,10]=\{s \in \mathbb{C}\ ; \frac{1}{2} \leq \Re(s) < \infty, \Im(s) \geq 10\}$.

To do this, we need to demonstrate three facts:

\begin{enumerate}
   \item $|F(s)| \to 0 $ uniformly, if $s \in K[\frac{1}{2},\infty,10]$ and $|s| \to \infty$.
    \item $\Re(F(\frac{1}{2}+ti)) \leq 0 \in \mathbb{R}$, if $t>0 .$
    \item $\Re(F(\sigma+10i)) \leq 0 $, if $\sigma>\frac{1}{2}.$
\end{enumerate}

To prove the first fact, initially, it is observed that using the expression of the function $\xi(s)$, we can write $F(s)$ as:

\begin{equation}
    F(s)=\frac{1}{(\frac{1}{2}-s)}\frac{\zeta(s)^{2}}{\{\psi(s)+(1+\psi(s))^{2}\}}
\end{equation}
where 
\begin{equation}
    \psi(s)=\pi^{s-\frac{1}{2}}\frac{\Gamma(\frac{1-s}{2})}{\Gamma(\frac{s}{2})}.
\end{equation}

Note that :

\begin{equation}
    \|\psi(\sigma+ti)\|=O(t^{\frac{1}{2}-\sigma})
\end{equation}
\newline
and as $\|\zeta(\sigma+ti)\|=O(t^{\frac{1}{4}})$, if $\sigma>\frac{1}{2}$ and $t>\pi$ , we have:

\begin{equation}
    \lim_{\|s\| \to \infty} F(s)=0.
\end{equation}
and the limit occurs uniformly.

To prove the validity of (2), we will use the initial expression of $F(s)$ in equation (4.6), observe that by Corollary (4.1) 

\begin{equation}
    \Im(\xi(\frac{1}{2}+ti))=0
\end{equation}

Hence:

\begin{equation}
    F(\frac{1}{2}+ti)=\frac{i}{t}\frac{\|\xi(\frac{1}{2}+ti)\|^{2}}{\{\pi^{-\frac{1}{2}}\left(\|\Gamma(\frac{1}{4}+\frac{ti}{2}\|\right)^{2}+\left(2 \pi^{-\frac{1}{4}}\Re(\Gamma(\frac{1}{4}+\frac{ti}{2})\cos(\frac{\pi}{4} t)\right)^{2}\}}
\end{equation}
(Where the fact $\Gamma(\sigma-ti) = \Gamma(\sigma+ti)^{*}$ was used.)
\newline
and

\begin{equation}
    \Re(F(\frac{1}{2}+ti)=0
\end{equation}
thus proving its validity.

It remains to prove the validity of condition (3), for this, we use expression (4.7) 

\begin{equation}
    F(s)=\frac{1}{(\frac{1}{2}-s)}\frac{\zeta(s)^{2}}{\{\psi(s)+(1+\psi(s))^{2}\}}
\end{equation}
where 
\begin{equation}
    \psi(s)=\pi^{s-\frac{1}{2}}\frac{\Gamma(\frac{1-s}{2})}{\Gamma(\frac{s}{2})}.
\end{equation}
For $\sigma$ "sufficiently large," we have:

\begin{equation}
F(\sigma+10i)\approx \frac{1}{\frac{1}{2}-\sigma-ti}
\end{equation}
\newline
Certainly, as $\sigma \to \infty$, the remaining terms in (4.14) converge uniformly to $1$. Under this condition, requirement (3) is fully satisfied.

The examination of the real part of $F(\sigma+10i)$ for "small" $\sigma$ is performed by plotting expression (4.14) on Wolfram, and with that, it is verified that condition (3) holds true (see Figure 1,2).
\begin{figure}[h]
  \centering
  \includegraphics[width=1\textwidth]{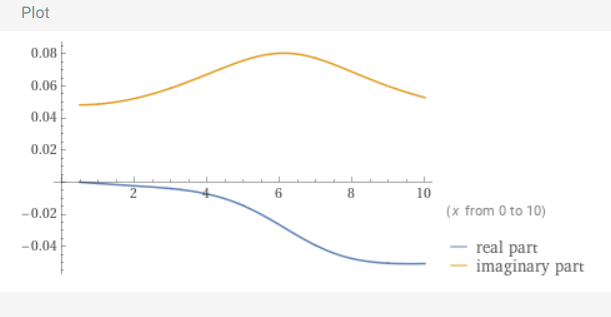}
  \caption{Graph of the real and imaginary parts of $F(\sigma+10i)$.}
  \label{fig:example}
\end{figure}

\begin{figure}[h]
  \centering
  \includegraphics[width=1\textwidth]{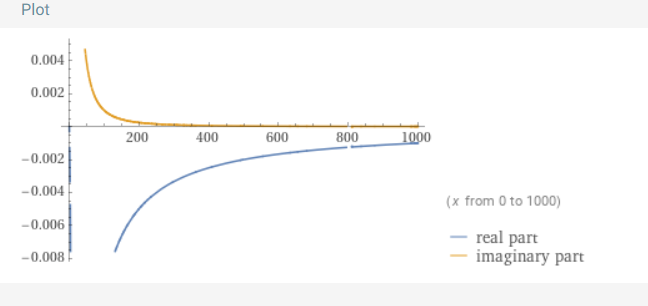}
  \caption{Graph of the real and imaginary parts of $F(\sigma+10i)$.}
  \label{fig:example}
\end{figure}
\end{proof}

\clearpage

\backmatter

\bibliography{sn-bibliography}

\end{document}